\documentclass[centertags,12pt]{amsart}
\usepackage{latexsym}
\usepackage{amsthm}
\usepackage{amssymb}
\usepackage{color}
\usepackage[utf8]{inputenc}
\usepackage{mathrsfs}
\usepackage[english]{babel}
\usepackage{setspace}
\usepackage{ textcomp }
\usepackage{graphicx}
\usepackage{float}
\usepackage{tikz}
\usepackage{soul}
\usepackage{amsmath}
\graphicspath{ {images/} }
\usepackage{hyperref}

\numberwithin{equation}{section}
\makeatletter
\renewcommand{\subsection}{\@startsection
	{subsection}{2}{0mm}{\baselineskip}{-0.25cm}
	{\normalfont\normalsize\bf}}
\makeatother

\newtheorem{theorem}{Theorem}[section]
\newtheorem{remark}[theorem]{Remark}
\newtheorem{proposition}[theorem]{Proposition}
\newtheorem{lemma}[theorem]{Lemma}
\newtheorem{corollary}[theorem]{Corollary}

{\theoremstyle{definition}

	}
\theoremstyle{remark}

\newcommand{\fqs}{{\mathbb F_{q^2}}}

\definecolor{Amaranto}{rgb}{0.9, 0.17, 0.31}
\definecolor{Borgogna}{rgb}{0.5, 0.0, 0.13}

\begin{document}
		\author[V. Pallozzi Lavorante]{Vincenzo Pallozzi Lavorante}
	\address{Dipartimento di Matematica Pura e Applicata, Universit\'a degli Studi di Modena e Reggio Emilia}
	\email{vincenzo.pallozzilavorante@unimore.it}
	\title{New families of permutation trinomials constructed by permutations of  $\mu_{q+1}$}
	
	\begin{abstract}
		Permutation polynomials are of particular significance in several areas of applied mathematics, such as Coding theory and Cryptography. Many recent constructions are based on the Akbary-Ghioca-Wang (AGW) criterion. Along this line of research, we provide new classes of permutation trinomials in $\mathbb{F}_{q^2}$ of the form $f(x)=x^r h(x^{q-1})$, by studying permutations of the set of $(q+1)$-th roots of unity, which look like monomials on the sets of suitable partitions.
	\end{abstract}
	
	\maketitle
	
	\begin{small}
		
		
		
	\end{small}

\maketitle
	\section{Introduction}
\label{intro}
Let $ \fqs $  be the finite field with $ q^2 $ elements, with $q=p^h$ a prime power. A polynomial $ f \in \fqs[x] $ is called a permutation polynomial of $ \fqs $ if its associated polynomial mapping $ f \colon c \mapsto f(c) $ from $ \fqs $ to itself is a bijection. \\
Permutation polynomials over finite fields have wide applications in coding theory, cryptography, and combinatorial design theory \cite{13}, and we refer to \cite{1,18} for more details of the recent advances and contributions to the area. For results in other areas we refer to \cite{8,7,14}.\\ 
In general, permutation polynomials without particular assumptions can be easily constructed. For example,  Lagrange interpolation provides an essential method to obtain permutation polynomials over a finite field.
However one can ask for additional properties to be satisfied by the polynomials, making this problem much more difficult. 

\noindent Recently, many authors focused on permutation binomials and trinomials. See  \cite{6,15,16,11,10}. 

\noindent In this paper we limit ourselves to the study of permutation trinomials over $\fqs$, with $q$ odd.
Every polynomial $p(x)$ over $\fqs$ has the form $a x^m fx^{(q^2-1)/n})+b$, whenever $f(0)=b$, for $m,n$ positive integer such that $n \mid q^2-1$. The smallest integer $n$ such that $p(x)=a x^m f(x^{(q^2-1)/n})+b$ is the \textit{index} of $ p(x) $.

\noindent In \cite{20} the authors provided an interesting way to obtain permutation trinomials of index $q+1$ over $ \fqs $ by studying permutations of $\mu_{q+1}$, the set of $q+1$-th roots of unity. They split $\mu_{q+1}$ into two or three disjoint subsets and studied permutations on these sets separately, by using monomials.
Following the same approach, we extend their method for a generic divisor $d$ of $q+1$ and study permutations on the sets $\mu_{\frac{q+1}{d}}$, in order to get bijections on $\mu_{q+1}$.
Under precise conditions, this gives new families of permutation trinomials over $\fqs$ of the type $f(x)=x^r h(x^{q-1})$.
Indeed, we construct permutation polynomials of type
\[x^r (c+x^{n_1(q-1)} \pm x^{n_2(q-1) } ),\]
where $n_1$ and $n_2$ are some properly chosen positive integers.

\noindent The paper is organized as follows.
In Section \ref{sec:2} we present some previous results that provides very useful methods for constructing permutations polynomial. In Subsection \ref{ss:pari} we deal with the case $d$ even and we prove the following Theorem.

\begin{theorem}\label{th:n1}
	Let $d$ be an even integer. Assume that $q \equiv -1 \pmod d$, $q \equiv 1 \pmod 4$ and $\gcd(\frac{q+1}{d},d)=1$. Let $c \in \fqs$ be such that $(\frac{c}{2})^\frac{q+1}{2} = 1$ and $k$ be an odd integer. Furthermore, let $v=\frac{q+1}{d}$ and $v_1=(d-1)v$. Let
	\[h(x)=c+x^{v_1+k}+x^{qv+k},\]
	then $f(x)=x^rh(x^{q-1})$ is a permutation polynomial of $\fqs$ if and only if $\gcd(r-k,\frac{q+1}{d})=1$ , $\gcd(r,q-1)=1$ and $\gcd(s+r-k,d)=1$, with $  s \equiv \frac{q+1}{d} \pmod d$ and $s\leq d$. 
\end{theorem}
Similarly, the case $d$ odd is treated in Subsection \ref{ss:pari}.
\begin{theorem}\label{th:n2}
	Let $d$ be an odd integer. Assume that $q \equiv -1 \pmod d$, $q \equiv 1 \pmod 4$ and $ \gcd\big(\frac{q+1}{d},d\big) =1 $. Let $c \in \fqs$ such that $(\frac{c}{2})^\frac{q+1}{2} = 1$ and $k$ an odd integer. Furthermore let $u = \frac{q+1}{d}+1$ and $u_1=(d-1)\frac{q+1}{d}+1$. Let
	\[h(x)=c+x^{u_1+k}+x^{qu+2+k},\]
	then $f(x)=x^rh(x^{q-1})$ is a permutation polynomial of $\fqs$ if and only if $\gcd(r-k-1,\frac{q+1}{d})=1$, $\gcd(r,q-1) = 1$ and $\gcd(s+r-k-1, d)=1$,
	with $  s \equiv \frac{q+1}{d} \pmod d  $, $s \leq d$.
\end{theorem}
At the end of Section \ref{sec:3}, we give two new constructions which hold for every $d$.
\begin{theorem}\label{th:n3}
	Let $q \equiv -1 \pmod d$, $c \in \mathbb{F}^*_{q^2}$. Furthermore write $v=\frac{q+1}{d}$ and $v_1=(d-1)v$. 
	Let \[h(x)=c+x^{v_1+k}-x^{qv+k},\]
	then $f(x)=x^rh(x^{q-1})$ is a permutation polynomial of $\fqs$ if and only if  $\gcd(r,q^2-1)=1$. 
\end{theorem}

\begin{theorem}\label{th:n4}
	Let $q \equiv -1 \pmod d$ and $c \in \fqs^*$. Furthermore let $u = \frac{q+1}{d}+1$ and $u_1=(d-1)\frac{q+1}{d}+1$. Let
	\[h(x)=c+x^{u_1+k}-x^{qu+2+k},\]
	then $f(x)=x^rh(x^{q-1})$, is a permutation polynomial of $\fqs$ if and only if  $\gcd(r,q^2-1) = 1$.
\end{theorem}
Section \ref{sec:4} is dedicated to some special cases, occurring when $d=2$ and $d=4$.
More precisely we obtain the following results.
\begin{theorem}\label{th:n5}
	Let  $q \equiv 3 \pmod 8$, and $c \in \fqs$ such that  $(\frac{c}{2})^{\frac{q+1}{4}}=1$. Furthermore, let $u=\frac{q+5}{4}$ and $k$ be an even integer. Let
	\[h(x)=c+x^{u+k}+x^{qu+k+2},\]
	then the polynomial $f(x)=x^rh(x^{q-1})$ is a permutation polynomial of $\fqs$ if and only if $\gcd(r,\frac{q^2-1}{4})=1$ and $\gcd(r-k-1,\frac{q+1}{4})=1$.
\end{theorem} 

\begin{theorem}\label{th:n6}
	Let $q \equiv 1 \pmod 2$ and $c \in \fqs$ such that  $(\frac{c}{2})^{\frac{q+1}{2}}=1$. Let
	\[h(x)=c+x^u(x^{\frac{q+1}{2}v}-1), \mbox{ with } 2 \nmid u, v\]
	then the polynomial $f(x)=x^rh(x^{q-1})$ is a permutation polynomial of $ \fqs $ if and only if $\gcd(r,\frac{q^2-1}{2})=1$ and $\gcd(r-u,\frac{q+1}{2})=1$.
\end{theorem}

M.\ Zieve has informed the author that the polynomials $f(x)$ in Theorems~1.1--1.4 are congruent
mod $(x^{q^2}-x)$ to polynomials whose permutation criteria was previously known, so that in fact the permutation
polynomials in these results are not new.  Specifically, the polynomials $f(x)$ in Theorems~1.3 and 1.4
are congruent to $cx^r$ mod~$(x^{q^2}-x)$, so that Theorems~1.3 and 1.4 follow from the well-known
permutation condition for monomials.  Likewise, the polynomials $f(x)$ in Theorems~1.1 and 1.2 are congruent
mod~$(x^{q^2}-x)$ to particular instances of the binomials in \cite[Cor.~5.3]{ZR}, so that Theorems~1.1 and 1.2
follow immediately from \cite[Cor.~5.3]{ZR}.

\section{Previous results}\label{sec:2}
Here we recall the basic strategy. The following theorem follows from a
well-known folklore \cite{PL,Zieve}.
Let $\mu_{q+1}$ be the set of $q+1$-th roots of unity.
\begin{theorem}\label{cr:1}
	Let $h(x) \in \fqs[x]$. Then $f(x)=x^r h(x^{q-1})$ is a permutation polynomial if and only if the followings are satisfied:
	\begin{enumerate}
		\item $ \gcd(r,q-1)=1 $;
		\item $ x^r h(x)^{q-1} $ permutes $\mu_{q+1}$.
	\end{enumerate}
\end{theorem} 
Thus, in order to get permutation polynomials of $\fqs$ we need to find permutations of $\mu_{q+1}$. 
We focus our attention on the case where $h(x)$ is a trinomial. 

\noindent In \cite{20} the authors presented a useful tool to investigate permutations of $\mu_{q+1}$. 
Let $d$ a positive integer such that $d\mid q+1$, $\xi$  be a primitive $ d $-th root of unity, $  S_0=\mu_{\frac{q+1}{d}} $ and $S_i=\xi^i S_0$ for $i \leq 1 \leq d-1$.

\begin{proposition}{\cite[Theorem 1.2]{20}}\label{1.1}
	Let $d$ be a positive integer such that $d \mid q+1$ and let $A_i \in \mu_{q+1}$ for $0\leq i \leq d-1$. For $g(x) \in \mathbb{F}_{q^2}[x]$, if
	\[g(x)=A_i x^{r_i}, \mbox{ for } \ x \in S_i,\]
	then $g(x)$ is a permutation of $\mu_{q+1}$ if and only if the following hold
	\begin{enumerate}
		\item $\gcd(r_i,\frac{q+1}{d})=1$, for $0 \leq i \leq d-1$;
		\item $A_i x_i^{r_i} \neq A_j x_j^{r_j}$ for $x_i \in S_i$ and $x_j \in S_j$.
	\end{enumerate} 
\end{proposition}
Proposition \ref{1.1} gives a new approach to construct permutations of $\mu_{q+1}$ via monomials.
In \cite{20} it was used when $d=2$ and $d=3$, providing new classes of permutation polynomials.

\noindent In this paper, we extend this approach to the case $d > 3$. In fact, we  split $\mu_{q+1}$ into $d$ disjoint subsets using a primitive $ d $-th root of unity and we give conditions so that we can apply Result \ref{1.1} to our case. See Theorem \ref{th:omega}.

\section{New permutation trinomials over $\fqs$}\label{sec:3}

In this section we present the main results of this paper.
\begin{center}
	From now on $d$ is a positive integer such that $d \mid q+1$.
\end{center}
First, we need conditions which allows us to use Result \ref{1.1}.
Let $d$ be a positive integer such that $\gcd(d,\frac{q+1}{d})=1$ and consider a $d$-th primitive root of unity $\omega$.
We can split $\mu_{q+1}$ as $\mu_{q+1}=\bigcup_{i=0}^{d-1} \omega^i \mu_{\frac{q+1}{d}}$, with the subsets mutually disjoint. Indeed,
\begin{remark}
	The crucial point is to establish when the classes $ \omega^i \mu_{\frac{q+1}{d}} $ are all pairwise disjoint. Thus, we note that the subgroup generated by $\omega$ has order $d$, while $\mu_{\frac{q+1}{d}}$ has order $\frac{q+1}{d}$. This implies that  \[\langle \omega \rangle \cap \mu_{\frac{q+1}{d}} =\{1\} \iff \gcd\bigg(\frac{q+1}{d},d\bigg) =1. \]
	and we need to add this further condition.
\end{remark}
After that, we use the following Theorem as the key to prove permutations on $\mu_{q+1}$.
\begin{theorem}\label{th:omega}
	Let $ d $ be a positive integer such that $\gcd(d, \frac{q+1}{d})=1$ and $g(x) \in \fqs[x]$ be a monomial such that 
	\[g(x)=A \omega^{i k} y^n, \mbox{ if } x= \omega^i y \in \omega^i\mu_{\frac{q+1}{d}}\]
	with $A \in \mu_{q+1}$. Then, $g(x)$ is a permutation of $\mu_{q+1}$ if and only if 
	\begin{itemize}
		\item[(i)] $\gcd(\frac{q+1}{d},n )=1$;
		\item[(ii)] $\gcd(k,d)=1$.
	\end{itemize}
\end{theorem}
\begin{proof}
	Note that the first condition of Result \ref{1.1} is satisfied since $\gcd(n,\frac{q+1}{d})=1$. On the other hand, choose $i < j <d$. Then $ A \omega^{i k} y^n \ne A \omega^{j k} y^n $ if and only if 
	\[\omega^{(j-i)k} \notin \mu_{\frac{q+1}{d}},\]
	which is equivalent to
	\[\omega^{(j-i)k\frac{q+1}{d}}\ne 1. \]
	Since $\omega^r = 1$ if and only if $d \mid r$, we obtain
	\[\omega^{(j-i)k\frac{q+1}{d}} \ne 1 \iff d \nmid (j-i)k, \]
	which means $\gcd(k,d)=1$. In fact $j-i$ runs over any positive integer less than $d$.
\end{proof}

\subsection{Proof of Theorem \ref{th:n1}}\label{ss:pari}

The proof is obtained using a preliminary Lemma.
\begin{center}
	Let $d$ be an even integer.
\end{center}
Fix a primitive $d$-th root of unity, say $\omega$.

\begin{lemma}\label{lm:1}
	Let $v=\frac{q+1}{d}$ and $v_1=(d-1)v$. Then
	\begin{equation}\label{v}
		x^{v_1+k}=x^{qv+k},
	\end{equation}
	for every $x \in \mu_{q+1}$. Furthermore $x^{-v_1}=\omega^{i s}$, if $ x \in \omega^i \mu_{\frac{q+1}{d}}$ and $s \equiv v \pmod d$.
\end{lemma}
\begin{proof}
	Every $x \in \mu_{q+1}$ can be written as $x=\omega ^i y$, with $y \in \mu_{\frac{q+1}{d}}$. This implies that
	\[x^{(d-1)\frac{q+1}{d}}=\omega^{(d-1)is}=\omega^{-is}\]
	and 
	\[x^{qv}=\omega^{q i s}= \omega^{-is}\]
	where we used that $\omega^{q+1}=\omega^d=1$.
\end{proof}

\begin{proof}[Theorem \ref{th:n1}]
	By Corollary \ref{cr:1}, $f(x)$ permutes $ \fqs $ if and only if $\gcd(r,q-1)=1$ and the polynomial 
	\[g(x)=x^rh(x)^{q-1}\]
	permutes $\mu_{q+1}$. For every $x \in \mu_{q+1}$ we have 
	\[g(x)=x^r(c+2x^{v_1+k})^{q-1}=x^r\frac{c^q+2x^{-v_1-k}}{c+2x^{v_1+k}}.\]
	Note that $v_1+k$ is an even integer, as $k$ is odd and $d$ even. Moreover $q \equiv 1 \pmod4$ implies that $ (-1)^\frac{q+1}{2}=-1 $. Thus, from $ \big(\frac{c}{2}\big)^\frac{q+1}{2}=1 $, we obtain $c \ne  -2 x^{v_1+k}$ for every $x \in \mu_{q+1}$.
	Together with Lemma \ref{lm:vev1}, this leads us to the following simplification:
	\[g(x)=\frac{2}{c}x^{-v1}x^{r-k}=\frac{2}{c}\omega^{is}(\omega^i y)^{r-k}=\frac{2}{c}\omega^{i(s+r-k)}y^{r-k},\]
	if $x \in \omega^i \mu_{\frac{q+1}{d}}$ and $x=\omega^i y$, $y \in \mu_{\frac{q+1}{d}}$.
	Note that  $\frac{2}{c} \in \mu_{q+1}$.
	From Theorem \ref{th:omega},  $g(x)$ is a bijection of $\fqs$ if and only if
	\[ \gcd(r-k,v)=1 \mbox{ and }\gcd(s+r-k,d)=1.\]
	Combining that with $\gcd(q-1,r)=1$ we complete the proof.
\end{proof}

\subsection{Proof of Theorem \ref{th:n2} } \label{ss:dispari}

The proof is obtained using a preliminary Lemma.
\begin{center}
	Let $d$ be an odd integer.
\end{center}
Fix a primitive $d$-th root of unity, say $\omega$.

\begin{lemma}\label{lm:2}
	Let $u = \frac{q+1}{d}+1$ and $u_1=(d-1)\frac{q+1}{d}+1$.
	Then
	\begin{equation}\label{eq:u}
		x^{u_1+k} = x^{qu+k+2},
	\end{equation}
	for every $x \in \mu_{q+1}$. Furthermore $x^{-u_1}=\omega^{is}x^{-1}$, if $x \in \omega^i \mu_{\frac{q+1}{d}}$ and $s \equiv \frac{q+1}{d} \pmod d$.
\end{lemma}
\begin{proof}
	The proof is obtained from the one of Lemma \ref{lm:1} noting that $u=v+1$ and $u_1=v_1+1$.
\end{proof}

\begin{proof}[Proof of Theorem \ref{th:n2}]
	By Corollary \ref{cr:1}, $f(x)$ permutes $\fqs$ if and only if $\gcd(r,q-1)=1$ and the polynomial defined by
	\[g(x)=x^r h(x)^{q-1}\]
	permutes $\mu_{q+1}$.\\
	Using Lemma \ref{lm:2}, for every $x \in \mu_{q+1}$ we have:
	\begin{equation}
		g(x)=x^r(c+2x^{u_1+k})^{q-1}= x^r \frac{c^q+2x^{-u_1-k}}{c+2x^{u_1+k}}.
	\end{equation}
	Since $ \big(\frac{c}{2}\big)^\frac{q+1}{2}=1 $ and $ (-x^{u_1+k})^\frac{q+1}{2}=-1 $, from the assumptions that $q \equiv 1 \pmod 4$ and $k,d$ odd, we cannot have $c = - 2x^{u_1+k}$.
	This allows us to simplify as follows:
	
	\begin{equation}
		g(x)=\frac{2}{c} x^{-u_1} x^{r-k}=\frac{2}{c} \omega^{is} x^{r-k-1}, \mbox{ for } x \in \omega^i\mu_{\frac{q+1}{d}},
	\end{equation}
	where $\frac{q+1}{d} \equiv s \pmod d $. \\
	Note that $\frac{2}{c} \in \mu_{q+1}$ and let $x=\omega^i y $, with $y \in \mu_{\frac{q+1}{d}}$. This implies that
	\begin{equation}
		g(x)=\frac{2}{c} \omega^{i(s+r-k-1)} y^{r-k-1} 
	\end{equation}
	Thus, according to Theorem \ref{th:omega},  $g(x)$ is a permutation of $\mu_{q+1}$ if and only if \[\gcd(r-k-1, \frac{q+1}{d})=1 \mbox{ and }  \gcd(s+r-k-1,d)=1.\]
	Combining these results with $ \gcd(r,q-1)=1 $ we complete the proof. 
\end{proof}

\subsection{Proof of Theorems \ref{th:n3} and \ref{th:n4}} \label{ss:general}
In this subsection we use equation \eqref{v} and \eqref{eq:u} in order to give other constructions of polynomials permuting $\fqs$. 
Let $d$ be any positive integer such that $d \mid q+1$.
\begin{proof}[Proof of Theorem \ref{th:n3}]$ $
	We have that \[g(x)=x^r h(x)^{q-1}\] is permuting $\mu_{q+1}$ if and only if $\gcd(r,q-1)=1$. Indeed, for every element $x \in \mu_{q+1}$, $g(x)$ can be simplified as 
	\[g(x)=c^{q-1}x^r.\]
	Thus, proof follows from Corollary \ref{cr:1}.
\end{proof}
Similarly we obtain the proof of Theorem \ref{th:n4}, which we omit.

\section{Other new constructions} \label{sec:4}
We now provide other constructions of permutation polynomials for $d=4$ and $d=2$. 
\subsection{Proof of Theorem \ref{th:n5}} 
Assume that $4 \mid q+1$ and let $\omega$ be a $4$-th primitive root of unity.
Recall that $\mu_{q+1}$ can be divided into four classes
\begin{itemize}
	\item $\mu_{\frac{q+1}{4}}:=\{a \in \fqs | a^{\frac{q+1}{4}}= 1 \} $
	\item  $-\mu_{\frac{q+1}{4}}:=\{a \in \fqs | a^{\frac{q+1}{4}}= -1\}$
	\item  $\mu'_{\frac{q+1}{4}}:=\{a \in \fqs | a^{\frac{q+1}{4}}= \omega \mbox{ and } \omega^2=-1\}$
	\item $-\mu'_{\frac{q+1}{4}}:=\{a \in \fqs | a^{\frac{q+1}{4}}= -\omega \mbox{ and } \omega^2=-1\}$
\end{itemize}
\begin{lemma}
	Let $u = \frac{q+5}{4}=\frac{q+1}{4}+1$. Then for $x \in \mu_{q+1}$ we have
	\begin{itemize}
		\item $ x^u = x ^{2+qu}=x ^{2-u} $, if $x \in \pm\mu_{\frac{q+1}{4}}$.
		\item $ x^u = -x ^{2+qu}=-x ^{2-u} $, if $x \in \pm\mu_{\frac{q+1}{4}}'$.
	\end{itemize}
\end{lemma}
\begin{proof}
	We use that $\omega^q = -\omega$ and $\frac{1}{\omega}=-\omega$. 
\end{proof}

\begin{proof}[Proof of Theorem \ref{th:n5}]	By Corollary \ref{cr:1}, $f(x)$ permutes $\fqs$ if and only if $\gcd(r,q-1)=1$ and the polynomial defined by $ g(x)=x^rh(x)^{q-1} $ permutes $\mu_{q+1}$. In this case we have
	\begin{itemize}
		\item $ g(x)=x^r(c+2x^kx^u)=x^r\frac{c^q+2x^{-u}x^{-k}}{c+2x^ux^k} $ if $x \in \pm\mu_{\frac{q+1}{4}}$.
		\item $g(x)=c^{q-1}x^r$ if $x \in \pm\mu_{\frac{q+1}{4}}'$.
	\end{itemize}
	In the first case, since $(\frac{c}{2})^{\frac{q+1}{4}}=1$ and $(-x^{u+k})^{\frac{q+1}{4}}=(-1)^{\frac{q+1}{4}}$ with $ \frac{q+1}{4} $ an odd integer, we have $c+2x^{u+k} \ne 0 $. Hence $g(x)$ can be simplified as
	\[g(x)=\begin{cases}
		\begin{aligned}
			&\frac{2}{c}x^{r-k-1}, x \in \mu_{\frac{q+1}{4}}\\
			-&\frac{2}{c}x^{r-k-1}, x \in -\mu_{\frac{q+1}{4}}
		\end{aligned}
	\end{cases},\] which permutes $\pm\mu_{\frac{q+1}{4}}$ if and only if $\gcd(r-k-1,\frac{q+1}{4})=1$.\\
	Similarly in the second case we can simplify $g(x)$ as $g(x)=\pm (\frac{2}{c})^2 \omega ^r x^r$, for $x \in \pm \mu_{\frac{q+1}{4}}'$.
	Then $g(x)$ permutes $\pm \mu_{\frac{q+1}{4}}'$ if and only if $\gcd(r,\frac{q+1}{4})=1$.\\
	Combining these results we obtain that $f(x)$ is a permutation polynomial of $\fqs$ if and only if $\gcd(r,\frac{q^2-1}{4})=1$ and $\gcd(r-k-1,\frac{q+1}{4})=1$.
\end{proof}

\subsection{Proof of Theorem \ref{th:n6}}
Assume now that $2 \mid q+1 $.
\begin{proof}[Theorem \ref{th:n6}]In this case the polynomial $g(x) = x^r h(x)^{q-1}$ is defined by
	\begin{equation}
		g(x)=\begin{cases}
			\begin{aligned}
				&c^{q-1}x^r,&x \in \mu_{\frac{q+1}{2}} \\
				&x^r(c-2x^u)^{q-1}, &x \in -\mu_{\frac{q+1}{2}}
			\end{aligned}
		\end{cases}
	\end{equation}	
	Since $(\frac{c}{2})^{\frac{q+1}{2}}=1$, $g(x)$ permutes $\mu_{\frac{q+1}{2}}$ if and only if $\gcd(r,\frac{q+1}{2})=1 $.\\
	Furthermore for every $x \in -\mu_{\frac{q+1}{2}}$ we have 
	\[c-2x^u \neq 0 \iff 2 \nmid u \]
	This implies that we can simplify $g(x)$ as $g(x)=(-1)^{r-u}\frac{2}{c}x^{r-u}$, which permutes $-\mu_{\frac{q+1}{2}}$ if and only if $r$ is odd and $\gcd(r-u,\frac{q+1}{2})=1$.\\
	Combining these results, we obtained that $f(x)$ is a permutation polynomials of $\fqs$ if and only if $\gcd(r-u,\frac{q+1}{2})=1$ and $\gcd(r,\frac{q^2-1}{2})=1$.
\end{proof}
Note that for $v$ an even integer we have the standard case $f(x)=c x^r$ for $c \in \fqs$.

\begin{corollary}
	Let $ q \equiv 3 \pmod 4 $ and $(\frac{c}{2})^{\frac{q+1}{2}}=1$.\\
	The polynomial $f(x)=x^rh(x^{q-1})$, where
	\[h(x)=c+x^u(x^{\frac{q+1}{2}v}-1), \mbox{ with } 2 \nmid u, v\]  
	is not a permutation polynomial of $\fqs$.
\end{corollary}

%

%
\section*{Conflict of interest}

The authors declare that they have no conflict of interest.



\end{document}